\documentclass[11pt,a4paper]{article}

\usepackage{amsfonts}
\usepackage{amsmath}
\usepackage{amssymb}
\usepackage{amsthm}
\usepackage{graphicx, float, url}
\usepackage{mdwlist}
\usepackage[ruled]{algorithm2e}
\usepackage{setspace}

\input amssym.def
\input amssym.tex

\newcommand{\G}{\Gamma}
 
\newcommand{\Z}{\mathbb{Z}}

\newcommand{\Aut}{\mathrm{Aut}}

\def\Ga{\Gamma}

\def\a{\alpha}

\def\s{\sigma}

\def\la{\langle}
\def\ra{\rangle}

\def\Cay{{\rm Cay}}
\def\AGL{{\rm AGL}}
\def\ord{{\rm ord}}

\def\int{\cap}

\def\choose#1#2{\left (\!\!\begin{array}{c}#1\\#2\end{array}\!\!\right )}
\long\def\delete#1{}

\newtheorem{theorem}{Theorem}[section]
\newtheorem{corollary}[theorem]{Corollary}
\newtheorem{lemma}[theorem]{Lemma}

\newtheorem{example}[theorem]{Example}

\newcommand{\be}{\begin{equation}}
\newcommand{\ee}{\end{equation}}
\newcommand{\bea}{\begin{eqnarray}}
\newcommand{\eea}{\end{eqnarray}}
\newcommand{\bean}{\begin{eqnarray*}}
\newcommand{\eean}{\end{eqnarray*}}

\def\non{\nonumber}

\title{Rotational circulant graphs}
\author{Alison Thomson and
Sanming Zhou\\ \\
Department of Mathematics and Statistics\\
The University of Melbourne\\
Parkville, VIC 3010, Australia\\
\small{\it smzhou@ms.unimelb.edu.au}}

\date{August 23, 2013}

\begin{document}
\openup 0.5\jot 
\maketitle

\begin{abstract}
A Frobenius group is a transitive permutation group which is not regular but only the identity element can fix two points. Such a group can be expressed as the semi-direct product $G = K \rtimes H$ of a nilpotent normal subgroup $K$ and another group $H$ fixing a point. A first-kind $G$-Frobenius graph is a connected Cayley graph on $K$ with connection set an $H$-orbit $a^H$ on $K$ that generates $K$, where $H$ has an even order or $a$ is an involution. It is known that the first-kind Frobenius graphs admit attractive routing and gossiping algorithms. A complete rotation in a Cayley graph on a group $G$ with connection set $S$ is an automorphism of $G$ fixing $S$ setwise and permuting the elements of $S$ cyclically. It is known that if the fixed-point set of such a complete rotation is an independent set and not a vertex-cut, then the gossiping time of the Cayley graph (under a certain model) attains the smallest possible value. In this paper we classify all first-kind Frobenius circulant graphs that admit complete rotations, and describe a means to construct them. This result can be stated as a necessary and sufficient condition for a first-kind Frobenius circulant to be 2-cell embeddable on a closed orientable surface as a balanced regular Cayley map. We construct a family of non-Frobenius circulants admitting complete rotations such that the corresponding fixed-point sets are independent and not vertex-cuts. We also give an infinite family of counterexamples to the conjecture that the fixed-point set of every complete rotation of a Cayley graph is not a vertex-cut.  


\end{abstract}

\section{Introduction}
\label{sec:int}

\subsection{Motivation}
\label{subsec:motiv}

Given a group $G$ with identify element $1$ and a subset $S$ of $G$ with $1 \not \in S$ and $S = S^{-1} := \{s^{-1}: s \in S\}$, the {\em Cayley graph} on $G$ with {\em connection set} $S$, denoted $\Cay(G, S)$, is the graph with vertex set $G$ such that $g, h \in G$ are adjacent if and only if $gh^{-1} \in S$. A Cayley graph on a cyclic group is called a {\em circulant graph} (or  {\em circulant} for short), or a multi-loop network as used in computer science.  

It is widely recognized \cite{AK, ABR, CF, Hey} that Cayley graphs provide good models for interconnection networks, and as such they have been studied extensively in computer science for more than two decades. In particular, since the 1970s \cite{WC}, circulant graphs have attracted attention due to their regularity, simple structure and rich symmetry. We refer the reader to \cite{BCH, Hwang1, Hwang} for surveys on circulant graphs in the context of network design with an emphasis on connectivity, diameter and efficiency in information dissemination. The reader may also consult \cite{Hey, LJD} for a survey on Cayley graphs as models for interconnection networks and an account of popular Cayley networks such as hypercubes, cube-connected cycles, star graphs, and so on.  

A. \textsc{Complete rotations.}~~Motivated by the need to construct fast gossiping algorithms, Bermond, Kodate and P\'erennes introduced \cite{BKP} the concept of complete rotation in Cayley graphs. (See \S \ref{subsec:def} for terminology not defined in this subsection.) Among other things they proved \cite[Corollary 15]{BKP} that, if a Cayley graph with $n$ vertices and valency $d$ admits a complete rotation whose set of fixed points is empty, then its gossiping time under the store-and-forward, all-port and full-duplex model \cite{BKP} is equal to the trivial lower bound $(n-1)/d$, which is the best one can hope for. They also proved \cite[Lemma 17]{BKP} that, if a Cayley graph admits a complete rotation whose fixed-point set is an independent set and not a vertex-cut, then its gossiping time is equal to $\lceil (n-1)/d \rceil$. (A {\em vertex-cut} of a graph is a subset of its vertex set whose removal increases the number of connected components.) They proved further that several popular network structures, including hypercubes, star graphs and multi-dimensional tori, admit complete rotations, and using this they determined the exact value of the gossiping time of these networks. In \cite{FA}, Fragopoulou and Akl introduced a similar notion of rotation and used it to construct efficient communication algorithms. What is called a rotation of a Cayley graph in \cite{FA} is a complete rotation (as used in \cite{BKP, HMP} and the present paper) that is also an inner automorphsim of the group defining the Cayley graph. Since complete rotations enable efficient and simple gossiping algorithms \cite{BKP, FA}, an interesting but challenging problem, as addressed in \cite{BKP}, is to classify or characterize Cayley graphs that admit a complete rotation. This problem was studied by Heydemann, Marlin and  P\'{e}rennes in \cite{HMP}, where they gave among other things a few group-theoretic conditions for the existence of a complete rotation in a Cayley graph. For example, they proved \cite[Corollary 3.1]{HMP} that a connected Cayley graph $\Cay(G, S)$ admits a complete rotation if and only if there exists a presentation $\la S | R \ra$ of $G$ such that the free group $F(S)$ has an automorphism that fixes $R$ and induces a cyclic permutation of $S$. Proving a conjecture of Bermond, Kodate and P\'erennes \cite{BKP}, Lichiardopol \cite{Lich} showed that the $k$-dimensional toroidal mesh with $n^k$ vertices (where $n \ge 3$) admits a complete rotation whose fixed-point set is not a vertex-cut. In the same paper, he also disproved a stronger conjecture by the same group \cite{BKP} which asserts that the fixed-point set of any complete rotation of a Cayley graph is not a vertex-cut.  

B. \textsc{Balanced regular Cayley maps.}~~It was observed in \cite{LP} that a complete rotation in a Cayley graph is a special skew-morphism \cite{JS}, and that the existence of a complete rotation in a Cayley graph $\Ga$ is equivalent to the existence of a 2-cell embedding of $\Ga$ on a closed orientable surface as a balanced regular Cayley map \cite{SS}. Regular Cayley maps themselves are important objects of study in topological graph theory. In particular, from \cite{JS, SS} it is evident that constructions and/or characterizations of balanced regular Cayley maps are of considerable interest to the field of regular maps. 

C. \textsc{Frobenius graphs and HARTS (hexagonal meshes).}~~In \cite[Theorem 5.3]{Z}, the second author proved that, if there exists a subgroup $H$ of the setwise stabilizer $\Aut(G, S)$ of $S$ in $\Aut(G)$ such that the union of the $H$-orbits on $G \setminus \{1\}$ of length less than $|H|$ is an independent set and not a vertex-cut of $\Ga = \Cay(G, S)$, then the gossiping time of $\Ga$ is given by $\lceil (n-1)/d \rceil$. He also proved \cite[Theorem 5.1]{Z} that the same result holds if there exists a subgroup $H \le \Aut(G)$ that is regular on $S$ and semiregular on $G \setminus \{1\}$. Such graphs are called the first-kind Frobenius graphs, and they were studied in \cite{Sole, FLP} and further investigated in \cite{Z} from a communication point of view. These two results in \cite{Z} generalize Lemma 17 and Corollary 15 in \cite{BKP}, respectively. In \cite[Theorem 5.1]{Z} it was also proved that the first-kind Frobenius graphs admit `perfect' optimal gossiping algorithms in some sense. Moreover, by \cite{Sole, FLP} and \cite[Section 6]{Z}, such graphs achieve the smallest possible edge-forwarding \cite{HMS} and arc-forwarding indices \cite{Z} and also support `perfect' routing schemes in a sense. It is thus desirable to classify, for example, the first-kind Frobenius circulants, and this has been achieved in \cite{TZ-1} and \cite{TZ-2} in the cases of valency 4 and 6, respectively. It turns out that these two families of circulants all admit complete rotations; see \cite[Theorem 2]{TZ-1} and \cite[Theorem 2]{TZ-2}, respectively. Moreover, recently we became aware of the fact that a very special subfamily of first-kind Frobenius circulants of valency 6 had physically been used \cite{CSK, DRS} as multiprocessor interconnection networks at the Real-Time Computing Laboratory, The University of Michigan. They are called HARTS ({\em Hexagonal Architecture for Real-Time Systems}) \cite{DRS} or {\em hexagonal mesh interconnection networks} \cite{ABF} in the computer science literature. The fact that the topological structures of HARTS are indeed some special first-kind Frobenius circulants of valency 6 was justified in \cite{TZ-2}; see also \S \ref{subsec:harts} in this paper. In the combinatorial community they were first studied in \cite{YFMA}, and their optimal routing and gossiping algorithms were given in \cite{TZ}. 

We refer the reader to \cite{FZ, Z1} for other recent results on Frobenius graphs, and to \cite{HKRU} for a survey on information dissemination in communication networks.

\subsection{A sumary of main results}
\label{subsec:cont}

From the above it should be clear that it would be interesting to classify all rotational Frobenius circulants and give a methodology for constructing them. We will achieve this in the present paper with the help of number theory; see Theorem \ref{bigcomposite}. In view of (B) above, this can also be viewed as a classification of Frobenius circulants that are 2-cell embeddable on closed oriented surfaces as balanced regular Cayley maps (see Corollary \ref{core:map}). We will also study a family of rotational circulants with $p^e$ vertices, where $p$ is an odd prime and $e \ge 3$ an integer, and give a necessary and sufficient condition for the corresponding fixed-point set to be a non-vertex-cut; see Theorem \ref{thm:q}. On the one hand, from this result we obtain a new family of rotational circulants whose gossiping time is equal to the aforementioned trivial lower bound. On the other hand, this result generalizes Lichiardopol's single counterexample \cite{Lich} to the conjecture of Bermond, Kodate and P\'erennes \cite{BKP} as mentioned in (A), to an infinite family of counterexamples. 

\subsection{Definitions}
\label{subsec:def}

Let $G$ be a group with identity element $1$, and let $V$ be a set. An {\em action} of $G$ on $V$ is a mapping $V \times G \rightarrow V, (v, g) \mapsto v^g$, such that $v^1 = v$ and $(v^g)^h = v^{gh}$ for $v \in V$ and $g, h \in G$. We use $v^G := \{v^g: g \in G\}$ to denote the {\em $G$-orbit} containing $v$ and $G_{v} := \{g \in G: v^g = v\}$ the {\em stabilizer} of $v$ in $G$. $G$ is {\em semiregular} on $V$ if $G_{v} = 1$ is the trivial subgroup of $G$ for all $v \in V$, {\em transitive} on $V$ if $v^G = V$ for some (and hence all) $v \in V$, and {\em regular} on $V$ if it is both transitive and semiregular on $V$. If a group $H$ acts on $G$ such that $(uv)^h = u^h v^h$ for any $u, v \in G$ and $h \in H$, then $H$ is said to act on $G$ as a group. In this case we use $G \rtimes H$ to denote the semidirected product \cite{Dixon-Mortimer} of $G$ by $H$ with respect to the action.

It is well known that a Cayley graph $\Ga = \Cay(G, S)$ has valency $|S|$, and it is connected if and only if $\la S \ra = G$. Since $(x, g) \mapsto xg$, $x, g \in G$, defines a regular action of $G$ on $G$ (as a set) which preserves the adjacency of $\Ga$ (see e.g.~\cite{Biggs}), we may view $G$ as a subgroup of the automorphism group $\Aut(\Ga)$ of $\Ga$. Define $\Aut(G, S) := \{\a \in \Aut(G): S^{\a} = S\}$ to be the {\em setwise stabilizer} of $S$ in $\Aut(G)$ under the natural action of $\Aut(G)$ on $G$, and $\Aut(\Ga)_1$ the stabilizer of the vertex $1$ in $\Aut(\Ga)$. Then $\Aut(G, S) \le \Aut(\Ga)_{1}$ (see e.g.~\cite{Godsil} or \cite[Proposition 16.2]{Biggs}).

A bijection $\omega: G \rightarrow G$ is called \cite{BKP} a \emph{complete rotation} of 
$\Cay(G, S)$ if there exists an ordering of $S = \{s_0, s_1,
\ldots, s_{d-1}\}$ (where $d = |S|$) such that $\omega(1) = 1$ and $\omega(gs_i) = \omega(g)s_{i+1}$ for all $g \in G$ and $i = 0, 1, \ldots, d-1$, with subscripts mod $d$. In particular, $\omega(s_i) = s_{i+1}$ for each $i$ and so $\omega$ permutes the elements
of $S$ cyclically. In \cite[Proposititon 2.2]{HMP} it is shown that a bijection $\omega: G
\rightarrow G$ is a complete rotation of $\Cay(G, S)$ if and only if
$\omega \in \Aut(G, S)$ and for some (and hence all) $s \in S$, $s^{\la \omega
\ra} = \{s, s^{\omega}, s^{\omega^2}, \ldots, s^{\omega^{d-1}}\} = S$, where $\la \omega
\ra$ is the cyclic group generated by $\omega$. An element $g \in G$ is called a {\em fixed point of $\omega$} if $g \ne 1$ and there exists $i \in \{1, \ldots, d-1\}$ such that $g^{\omega^{i}} = g$. A Cayley graph is called {\em rotational} if it admits a complete rotation. 

An {\em arc} of a graph is an ordered pair of adjacent vertices. A graph $\Ga$ is {\em $G$-arc-transitive} if $G \le \Aut(\Ga)$ and $G$ is transitive on the set of arcs of $\Ga$. $\Ga$ is {\em arc-transitive} if it is $\Aut(\Ga)$-arc-transitive. It is easy to see that any rotational Cayley graph is arc-transitive.

Given a generating set $S$ of a group $G$ and a cyclic permutation $\rho$ of $S$, a {\em Cayley map} $M = CM(G, S, \rho)$ \cite{JS, SS} is a 2-cell embedding of the Cayley graph $\Cay(G, S)$ on an orientable surface such that for each vertex $g \in G$, the cyclic permutation of the arcs $(g, sg)$, $s \in S$, induced by a fixed orientation of the surface coincides with $\rho$. $M$ is called {\em balanced} \cite{SS} if $\rho(s^{-1}) = \rho(s)^{-1}$ for every $s \in S$, and {\em regular} if its automorphism group is regular on the set of arcs of $\Cay(G, S)$.  
 
A \emph{Frobenius group} $G$ is a transitive group on
a set $V$ which is not regular on $V$ such that
the only element of $G$ which fixes two points of $V$ is the
identity element of $G$. It is well known (see
e.g.~\cite[p.86]{Dixon-Mortimer}) that a finite Frobenius group
$G$ has a nilpotent normal subgroup $K$, called the
\emph{Frobenius kernel}, which is regular on $V$. Hence $G = K \rtimes H$, where $H$ is the stabilizer of a point of $V$; each such group $H$ is called a \emph{Frobenius complement} of $K$ in $G$. Since $K$ is regular on $V$, we may identify $V$ with $K$ in such a way
that $K$ acts on itself by right multiplication, and we may choose $H$ to be the
stabilizer of $1$ so that $H$ acts on $K$ by
conjugation. Obviously, $H$ is semiregular on $K \setminus \{1\}$. Following \cite[Theorem 1.4]{FLP}, a \emph{$G$-Frobenius graph} is a connected Cayley graph $\Cay(K, S)$ such that, for some $a \in K$ with $\la a^H \ra = K$, (i) $S =  a^H$ if $|H|$ is even or $a$ is an involution, or (ii) $S=a^H \cup (a^{-1})^H$ if $|H|$ is odd and $a$ is not an involution.
Hereinafter $x^H := \{h^{-1}xh: h \in H\}$ is the $H$-orbit containing $x \in K$ under the action of $H$ on $K$ (by conjugation). Since $G$ is a Frobenius group, $H$ may be regarded as a subgroup of $\Aut(K)$. Thus $H \le \Aut(K, S) \le \Aut(\Ga)_1$ and consequently $G \le \Aut(\Ga)$. We call $\Cay(K, S)$ a \emph{first-kind} or {\em second-kind} $G$-Frobenius graph \cite{Z} according as whether $S$ is given by (i) or (ii).
 
We refer the reader to \cite{BM}, \cite{Dixon-Mortimer} and \cite{BW00, IR, NZ} for graph, group and number-theoretic terminology and notation, respectively.

\section{Classification of rotational first-kind Frobenius circulants}
\label{sec:class}

\subsection{Preparations}
\label{subsec:fc}
 
In this subsection we collect several preliminary results that will be used in the proof of our main results. 

Let $\Z_n$ be the additive group of integers mod $n$, where $n \ge 3$ is an integer.  
Let $\Z_n^* = \{[u]: 1 \le u \le n-1, \gcd(n,u) = 1\}$ be the multiplicative group of units of the ring $\Z_n$. Here the residue classes $[u]$ are mod $n$, and we may write $[u]_n$ in place of $[u]$ if there is a danger of confusion. We have $\Aut(\Z_n) \cong \Z_n^*$ and $\Z^*_{n}$ acts on $\Z_{n}$ by the usual multiplication: $[x][u] = [xu]$, $[x] \in \Z_{n}$, $[u] \in \Z^*_{n}$. The semidirect product $\Z_{n} \rtimes \Z^*_{n}$ acts on $\Z_{n}$ such that $[x]^{([y], [u])} = [(x+y)u]$ for $[x], [y] \in \Z_{n}$ and $[u] \in \Z^*_{n}$. Thus, for a subgroup $H$ of $\Z_n^*$, $\Z_n \rtimes H$ inherits from $\Z_{n} \rtimes \Z^*_{n}$ the same action on $\Z_n$. We use $[u]^{-1}$ to denote the inverse element of $[u]$ in $\Z^*_{n}$, and $\la [u] \ra$ to denote the cyclic subgroup of $\Z^*_{n}$ generated by $[u]$. The operation of $\Z_{n} \rtimes \Z^*_{n}$ is defined by $([x_1], [u_1])([x_2], [u_2]) = ([x_1]+[x_2][u_1]^{-1}, [u_1 u_2])$ for $([x_1], [u_1]), ([x_2], [u_2]) \in \Z_{n} \rtimes \Z^*_{n}$. Thus the inverse element of $([x], [u])$ in $\Z_{n} \rtimes \Z^*_{n}$ is $(-[xu], [u]^{-1})$.

The first statement in the following lemma was proved in \cite[Lemma 3.1]{HMP}, and the second one can be verified easily.  
 
\begin{lemma}
\label{formofS}
Let $\G = \Cay(\Z_n,S)$ be a connected circulant graph of valency $d = |S|$. Then $\G$ admits an element of $\Z_n^*$ as a complete rotation if and only if there exist integers $h$ and $s$ both coprime to $n$ such that $S = \{[sh^i] : i = 1, \ldots, d\}$ and $h$ has order $d$ mod $n$. Moreover, in this case we have $\Ga \cong \Cay(\Z_n,S')$, where $S' = \langle [h] \rangle = \{[h^i] : i = 1, \ldots, d\}$, and $[sx] \mapsto [x], [x] \in \Z_n$ defines an isomorphism between the two graphs.
\end{lemma}

Thus a circulant graph of $n$ vertices admits $[h] \in \Z_n^*$ as a complete rotation if and only if it is isomorphic to $\Cay(\Z_n, \langle [h] \rangle)$ with $\langle [h] \rangle$ closed under inverse elements. 

It is easy to see \cite[Section 2]{Z} that a connected Cayley graph $\Cay(G, S)$ is a first-kind $G \rtimes H$-Frobenius graph if and only if $H \le \Aut(G)$ is regular on $S$ and semiregular on $G \setminus \{1\}$. When $G = \Z_n$, this yields: 

\begin{lemma}
\label{fundamentalcirculant}  
Let $\G = \Cay(\Z_n,S)$ be a connected circulant graph.  
\begin{itemize}
\item[\rm (a)] If there exists a subgroup $H$ of $\Z_n^*$ that is regular on $S$ and semiregular on $\Z_n \setminus \{[0]\}$, then $\Z_n \rtimes H$ is a Frobenius group and $\G$ is a first-kind $\Z_n \rtimes H$-Frobenius graph. 
\item[\rm (b)] Conversely, if $\G$ is a first-kind $\Z_n \rtimes H$-Frobenius graph, where $H$ is a subgroup of $\Z_n^*$ such that $\Z_n \rtimes H$ is a Frobenius group, then $H$ is regular on $S$ and semiregular on $\Z_n \setminus \{[0]\}$. 
\end{itemize}
\end{lemma}	

It can happen that a circulant $\Ga$ is isomorphic to some $K \rtimes H$-Frobenius graph $\Cay(K, S)$, but $K$ is not a cyclic group and $H$ is not regular on $S$. For example, if $n = p^e$ is a prime power with $e \geq 2$, then $G = \Z_p^e \rtimes \Z_{n-1}^* \cong \AGL(1, n)$ is a Frobenius group and the complete graph $K_n = \Cay(\Z_p^e, \Z_p^e \setminus \{0\})$ is a $G$-Frobenius graph. This describes $K_n$ as a Cayley graph on a non-cyclic group. Note that $K_n$ is not a $\Z_n \rtimes H$-Frobenius graph for any $H \le \Z_n^*$, because $\Z_n^*$ cannot have any subgroup regular on $\Z_n \setminus \{[0]\}$. 
    
In the sequel we focus on the first-kind Frobenius circulants such that the underlying Frobenius group $\Z_n \rtimes H$ has kernel $\Z_n$. Slightly abusing terminology, we say that such a circulant has {\em kernel} $\Z_n$ and {\em complement} $H$.  

\begin{lemma}\label{evenandodd}
Let $\G$ be a first-kind Frobenius circulant with kernel $\Z_n$. Then $n$ is odd and $\G$ has even valency.
\end{lemma}

\begin{proof}
Let $\G = \Cay(\Z_n,S)$ be a first-kind $\Z_n \rtimes H$-Frobenius circulant. Then $H$ is regular on $S$ and semiregular on $\Z_n \setminus\{[0]\}$ by Lemma \ref{fundamentalcirculant}. Hence $|H| = |S|$ and $|H|$ divides $n-1$. 

Suppose the valency $|S|$ of $\Ga$ is odd. Then $|H|$ is odd and thus by the definition of a first-kind Frobenius graph, $S = [x]^H$ for an involution $[x] \in \Z_n$. Since $[x]$ is an involution, so is every element of $S$. Since $\Z_n$ contains at most one involution, it follows that $|S| = 1$ and so $\Ga$ is a disconnected graph. This contradiction shows that $|S|$ must be even, which implies that $n$ must be odd as $|S|$ is a divisor of $n-1$.
\end{proof}
 
\begin{lemma} 
\label{semiregular}
(\cite[Lemma 4]{TZ-1})
A subgroup $H$ of $\Z_n^*$ is semiregular on $\Z_n \setminus \{[0]\}$ if and only if $[h-1] \in \Z_n^*$ for every $[h] \in H \setminus \{[1]\}$.
\end{lemma}
 
The following is a refinement of Lemma \ref{semiregular} in the case when $H$ is cyclic. Given coprime integers $m$ and $a$, we say that {\em $a$ has order $k$ mod $m$} if $a \pmod m$ has order $k$ in $\Z_m^*$. 

\begin{lemma}\label{order}
Let $H = \langle [h] \rangle$ be a cyclic subgroup of $\Z_n^*$ with order $|H| = d$ (where $1 \le h \le n-1$). Then $H$ is semiregular on $\Z_n \setminus \{[0]\}$ if and only if for every prime factor $p$ of $n$, $h$ has order $d$ mod $p$. 
\end{lemma}

\begin{proof}
Suppose that $H$ is semiregular on $\Z_n \setminus\{[0]\}$. Then by Lemma \ref{semiregular} we have $\gcd(h^i-1,n) = 1$ for $i = 1, \ldots, d-1$. Consequently, for $i = 1, \ldots, d-1$ and each prime factor $p$ of $n$, we have $\gcd(h^i-1,p) = 1$ and so $h$ does not have order $i$ mod $p$. On the other hand, since $H$ is cyclic of order $d$, we have $h^{d} \equiv 1 \pmod n$ and so $h^{d} \equiv 1 \pmod p$. Therefore, for each prime factor $p$ of $n$, $h$ has order $d$ mod $p$.

Conversely, suppose that $h$ has order $d$ mod $p$ for every prime factor $p$ of $n$. Then $h^i - 1 \not \equiv 0 \pmod p$ and so $\gcd(h^i-1,p) = 1$ for $i = 1, \ldots, d-1$. Since this holds for every prime factor $p$ of $n$, we have $\gcd(h^i-1,n) = 1$ for $i = 1, \ldots, d-1$. This together with Lemma \ref{semiregular} implies that $H$ is semiregular on $\Z_n \setminus\{[0]\}$.
\end{proof}

A subset $S$ of a finite group $G$ is called \cite{MKP01} a \emph{$CI$-subset} of $G$ if for every subset $S'$ of $G$, $\Cay(G,S)$ and $\Cay(G,S')$ are isomorphic if and only if $S' = S^\sigma$ for some $\sigma \in \Aut(G)$. It was proved by Muzychuk, Klin and P\"{o}schel \cite[Theorem 5.2]{MKP01}, and Dobson and Morris \cite{DM} independently, that every $S \subset \Z_n^*$ is a $CI$-subset of $\Z_n$.  

\begin{lemma}\label{CIsubset}
Let $\G = \Cay(\Z_n,S)$ and $\G' = \Cay(\Z_n,S')$ be rotational first-kind Frobenius circulants such that the underlying Frobenius groups have kernel $\Z_n$ and complements $H, H' \le \Z_n^*$, respectively. Then $\G$ and $\G'$ are not isomorphic unless $H = H'$. 
\end{lemma}

\begin{proof} 
By Lemma \ref{formofS}, we may assume $[1] \in S$ without loss of generality. Since by Lemma \ref{fundamentalcirculant} $H$ is regular on $S$, we have $S = H$ (as sets). Similarly, we may assume $[1] \in S'$ so that $S' = H'$.   

Assume $\G \cong \G'$. Then $\sigma S = S'$ for some $\sigma \in \Z_n^*$ by the aforementioned result of \cite{MKP01}. Since $S = H$ and $S' = H'$, we have $\sigma H = H'$. Since $[1] \in H'$, there exists $[h] \in H$ such that $\sigma [h] = [1]$, that is, $[h]$ is the inverse of $\s$ in $\Z_n^*$. Since $[h]$ is in $H$, so is $\s$, and hence $H' = \sigma H = H$.  
\end{proof}

\subsection{Classification}
\label{subsec:class}

The main result in this section is Theorem \ref{bigcomposite} below, which both classifies rotational first-kind Frobenius circulants with kernel $\Z_n$, and describes a means to construct them. In the proof of this result we will use the fact that for any odd prime $p$ and any integer $e \ge 1$, a primitive root mod $p^e$ is also a primitive root mod $p$. (In fact, for a primitive root $\eta$ mod $p^e$ and every $a \in \{1, \ldots, p-1\}$, there exists an integer $m$ such that $\eta^m \equiv a \pmod {p^e}$. Hence $\eta^m \equiv a \pmod p$ and $\eta$ generates $\Z_p^*$.) Denote by $\varphi$ Euler's totient function. 

\begin{theorem}
\label{bigcomposite}
Let $n = p_1^{e_1}\ldots p_l^{e_l}$ be a positive integer in canonical prime factorization. Let $D = \gcd(p_1 -1, \ldots, p_l-1)$. There exists a rotational first-kind Frobenius circulant with kernel $\Z_n$ and valency $d$ if and only if $n$ is odd and $d$ is an even divisor of $D$. 
 		
Moreover, if $n$ is odd and $d$ is an even divisor of $D$, then there are precisely $\varphi(d)^{l-1}$ pairwise non-isomorphic rotational first-kind Frobenius circulants of valency $d$ with kernel $\Z_n$. Each of these circulants is isomorphic to $\Cay(\Z_n, \langle [h] \rangle)$ for some $[h] \in \Z_n^*$ of the form 
	\begin{equation}\label{CRT1}
		h = \sum_{i=1}^l \frac{n}{p_i^{e_i}}b_i h_i,
	\end{equation} 
where $b_i$ and $h_i$ are integers satisfying
\begin{equation}
\label{CRT3}
b_i (n/p_i^{e_i}) \equiv 1 \pmod {p_i^{e_i}} 
\end{equation}
and	
	\begin{equation}\label{CRT2}	 
		 h_i \equiv \eta_i^{m_i \varphi(p_i^{e_i})/d} \pmod {p_i^{e_i}}
	\end{equation} 
for a fixed primitive root $\eta_i$ mod $p_i^{e_i}$ and an integer $m_i$ coprime to $d$. Furthermore, each such circulant is a first-kind $\Z_n \rtimes \Z_{d}$-Frobenius graph that admits $[h]$ above as a complete rotation.  
\end{theorem}

\begin{proof}
\textsc{Construction:} Let $n$ be odd, and $d$ an even divisor of $D$. We prove that the circulants described in the final portion of the theorem are rotational first-kind Frobenius circulants with kernel $\Z_n$ and valency $d$. For each $i \in \{1, \ldots, l\}$, let $\eta_i$ be a primitive root mod $p_i^{e_i}$, that is, $\eta_i$ is a generator of the cyclic group $\Z_{p_i^{e_i}}^*$. Since $d$ is a divisor of $p_i -1$, it is also a divisor of $|\Z_{p_i^{e_i}}^*| = \varphi(p_i^{e_i}) = p_i^{e_i-1}(p_i-1)$. Thus $\eta_i^{\varphi(p_i^{e_i})/d}$ generates the unique subgroup of $\Z_{p_i^{e_i}}^*$ with order $d$. Consequently, all generators of this subgroup have the form
 \[
 h_i = \eta_i^{m_i\varphi(p_i^{e_i})/d}
 \]
for an integer $m_i$ coprime to $d$.
It is easily seen that $h_i$ has order $d$ mod $p_i^{e_i}$. Therefore, $h_i^{d} \equiv 1 \pmod {p_i^{e_i}}$.

Set 
\[h = \sum_{i=1}^l \frac{n}{p_i^{e_i}}b_i h_i,\]
where the integer $b_i$ satisfies $b_i (n/p_i^{e_i}) \equiv 1 \pmod {p_i^{e_i}}$ (that is, $b_i$ is the inverse element of $n/p_i^{e_i}$ in $\Z_{p_i^{e_i}}^*$) and $h_i$ is as above. As detailed in \cite[Section 2.5]{NZ}, the Chinese Remainder Theorem implies that $h$ satisfies 
$h^{d} \equiv 1 \pmod n$. Moreover, for each $i \in \{1, \ldots, l\}$, $h$ satisfies $h \equiv h_i \pmod {p_i^{e_i}}$. Thus $h$ has order $d$ mod $p_i^{e_i}$, and therefore $h$ has order $d$ mod $n$. 

We now use $h$ above to construct a rotational circulant graph. Define $H = \langle [h] \rangle$. Then $H$ is a cyclic subgroup of $\Z_n^*$ of order $d$. Let $S = H$ (as sets).  We claim that $S$ is a connection set of $\Z_n$ and hence defines a circulant. To justify this it suffices to show $[-1] \in S$ (note that $[0] \not\in S$). In fact, since $d$ is even, we can write $d = 2k$, where $k$ is a positive integer. The cyclic group $\Z_{p_i^{e_i}}^*$ contains a unique involution, namely, $\eta_i^{\varphi(p_i^{e_i})/2} \equiv -1 \pmod {p_i^{e_i}}$. Note that $m_i$ must be odd because it is coprime to the even integer $d$. Since $k= d/2$, we have $h^k \equiv h_i^k \equiv \eta_i^{m_i \varphi(p_i^{e_i})/2} \equiv -1 \pmod {p_i^{e_i}}$, and consequently $h^k \equiv -1 \pmod n$. Since $[h] \in S$, it follows that $[-1] \in S$ and so $S = -S$. Therefore, $\G = \Cay(\Z_n, S)$ is a well-defined circulant graph of valency $|S| = d$. It is evident that $H$ is regular on $S$ and $\G$ admits $[h]$ as a complete rotation.
 
Next we show that $H$ is semiregular on $\Z_n \setminus{[0]}$. As noted above,  for each $i \in \{1, \ldots, l\}$, $h$ satisfies $h \equiv h_i \pmod {p_i^{e_i}}$ and has order $d$ mod $p_i^{e_i}$. Note that $h_i \equiv \eta_i ^{m_i \varphi(p_i^{e_i})/d} \equiv \eta_i^{m_i p_i^{e_i-1}(p_i-1)/d} \pmod {p_i}$. Note also that $\eta_i$ is a primitive root mod $p_i$ because it is a primitive root mod $p_i^{e_i}$. Thus $\eta_i$ has order $p_i-1$ mod $p_i$ and consequently $\eta_i^{p_i^{e_i-1}} \equiv \eta_i \pmod {p_i}$. Therefore, $h_i \equiv \eta_i^{m_i (p_i-1)/d} \pmod {p_i}$. Since $d$ is a divisor of $p_i-1$, $\eta_i^{(p_i-1)/d}$ has order $d$ mod $p_i$; in other words, $\eta_i^{(p_i-1)/d}$ generates the unique subgroup of $\Z_{p_i}^*$ of order $d$. Since $\gcd(m_i,d) = 1$, we deduce that $\eta_i^{m_i (p_i-1)/d}$ also has order $d$ mod $p_i$. Since $h_i \equiv \eta_i^{m_i (p_i-1)/d} \pmod {p_i}$, this means that $h_i$ has order $d$ mod $p_i$. Therefore, $h$ ($\equiv h_i \pmod {p_i}$) has order $d$ mod $p_i$ for each  $i \in \{1, \ldots, l\}$. It then follows from Lemma \ref{order} that $H$ is semiregular on $\Z_n \setminus \{[0]\}$. Finally, since $H \cong \Z_{d}$, we conclude that $\G$ is a first-kind $\Z_n \rtimes \Z_{d}$-Frobenius graph of order $d$ that admits $[h]$ as a complete rotation. 

\medskip
\textsc{Completeness of construction:} We have already seen in Lemma \ref{evenandodd} that no first-kind Frobenius circulants with kernel $\Z_n$ exist when $n$ is even. 

We now show that, up to isomorphism, every rotational first-kind Frobenius circulant with kernel $\Z_n$ arises from the construction above. Let $\G = \Cay(\Z_n,S)$ be such a circulant with valency $|S| = d$, where $n$ is odd. Then by Lemma \ref{fundamentalcirculant} there exists a subgroup $H$ of $\Z_n^*$ that is regular on $S$ and semiregular on $\Z_n \setminus \{[0]\}$ such that $\Ga$ is a $\Z_n \rtimes H$-Frobenius graph. Moreover, as we saw in the proof of Lemma \ref{CIsubset}, we may assume without loss of generality that $S = H = \langle [h] \rangle$ (as sets), where $[h]$ is a complete rotation of $\Ga$. Thus $h$ has order $d$ mod $n$, and $d$ must be even by Lemma \ref{evenandodd}. Since $H$ is semiregular on $\Z_n \setminus \{[0]\}$, by Lemma \ref{order},  for each $i \in \{1, \ldots, l\}$, $h$ has order $d$ mod $p_i$ and consequently $d$ divides $|\Z_{p_i}^*| = p_i -1$. Since this holds for every $i$, it follows that $d$ divides $D$. Moreover, $h$ also has order $d$ mod $p_i^{e_i}$ for each $i \in \{1, \ldots, l\}$. This is because the order of $h$ mod $p_i^{e_i}$ is bounded above by the order of $h$ mod $n$, and below by the order of $h$ mod $p_i$, but both bounds are equal to $d$. Since the unique subgroup of $\Z_{p_i^{e_i}}^*$ with order $d$ is generated by $\eta_i^{\varphi(p_i^{e_i})/d} \pmod {p_i^{e_i}}$, where $\eta_i$ is a primitive root mod $p_i^{e_i}$, without loss of generality we may assume $h \equiv \eta_i^{m_i\varphi(p_i^{e_i})/d}  \pmod {p_i^{e_i}}$, where $m_i$ is an integer coprime to $d$. As shown in \cite[Section 2.5]{NZ}, by the Chinese Remainder Theorem, $h$ has expression (\ref{CRT1}) with $b_i$ and $h_i$ as given in (\ref{CRT3}) and (\ref{CRT2}), respectively.

\medskip
\textsc{Enumeration:} We now enumerate all possible $h$ in (\ref{CRT1}) and determine when the corresponding first-kind Frobenius circulants are isomorphic. Since each $\Z_{p_i^{e_i}}^*$ has exactly one subgroup of order $d$, without loss of generality we may fix the primitive root $\eta_i$ in (\ref{CRT2}). There are exactly $\varphi(d)$ values $h_i$ as in (\ref{CRT2}), and each of them corresponds to a different value of $m_i$. Since this is true for each $i \in \{1, \ldots, l\}$ and since each $b_i \pmod {p_i^{e_i}}$ is unique, we conclude that there are exactly $\varphi(d)^l$ different possibilities for $h$. 

We now demonstrate that different choices of $h$ can generate the same group $H$ in $\Z_n^*$. Since $H$ is cyclic of order $d$, it has exactly $\varphi(d)$ generators. We claim that each of them corresponds to a different $h$ in (\ref{CRT1}). Let $h_0$ be a fixed solution to (\ref{CRT1}) corresponding to positive integers $m_i$ coprime to $d$. Then by (\ref{CRT1}) and (\ref{CRT2}), we have $h_0 \equiv \eta_i^{m_i \varphi(p_i^{e_i})/d}  \pmod {p_i^{e_i}}$ for each $i \in \{1, \ldots, l\}$. Let $m_j$ be a positive integer coprime to $d$. Then $h_0^{m_j} \equiv \eta_i^{m_i m_j \varphi(p_i^{e_i})/d}  \pmod {p_i^{e_i}}$. Since $m_i m_j$ is coprime to $d$, a possible value for $h_i$ in (\ref{CRT2}) is $h_i \equiv \eta_i^{m_i m_j \varphi(p_i^{e_i})/d} \pmod {p_i^{e_i}}$. Therefore, $h = h_0^{m_j}$ is also a solution to (\ref{CRT1}). However, $h_0$ and $h_0^{m_j}$ generate the same subgroup $H$ of $\Z_n^*$, and therefore the same first-kind $\Z_n \rtimes H$-Frobenius circulant $\G$. (Recall that the connection set for $\G$ is defined by $S = H$ in our construction.) In other words, if $\G$ is a rotational first-kind $\Z_n \rtimes H$-Frobenius circulant of order $d$, then $H$ has $\varphi(d)$ generators, and each of them is a different solution $h$ to (\ref{CRT1}). Since there are $\varphi(d)^l$ possible values for $h$, we deduce that there are at most $\varphi(d)^{l-1}$ rotational first-kind Frobenius circulants of valency $d$ with kernel $\Z_n$. However, by Lemma \ref{CIsubset}, different subgroups $H$ give rise to non-isomorphic circulants.  
Therefore, there are exactly $\varphi(d)^{l-1}$ rotational first-kind Frobenius circulants of valency $d$ with kernel $\Z_n$. 
\end{proof}

The valency $d$ of $\Cay(\Z_n, \langle [h] \rangle)$ in Theorem \ref{bigcomposite} is strictly less than the smallest prime divisor of $n$. In particular, if $3$ is a divisor of $n$, then $d = 2$ and so the only rotational first-kind Frobenius circulant with kernel $\Z_n$ is the cycle of length $n$. 

We illustrate Theorem \ref{bigcomposite} by the following example.

\begin{example}
\label{ex:6253}
{\em Let $n = 6253 = 13^2 \times 37$, so that $p_1 = 13$, $p_2 = 37$, and $D = \gcd(12,36) = 12$. Choose $\eta_1 = \eta_2 = 2$, which is a primitive root mod $13$ as well as a primitive root mod $37$. By (\ref{CRT3}), we have $b_1 \equiv 37^{-1} \equiv 32 \pmod {13^2}$ and $b_2 \equiv (13^2)^{-1} \equiv 30 \pmod {37}$.   

The even divisors of $D$ are $d = 2,4,6$ and $12$, and respectively, they produce $\varphi(d) = 1, 2, 2$ and $4$ non-isomorphic rotational first-kind Frobenius circulants $\Cay(\Z_{6253}, S)$ with kernel $\Z_{6253}$. These circulants are listed in Table \ref{tab6253}, omitting the pairs $(m_1, m_2)$ that produce a circulant isomorphic to one already in the table.

\begin{table}
\begin{center}
  \begin{tabular}{l|l|l|l|l}
\hline 
$d$ & $(m_1, m_2)$  & $(h_1, h_2)$ & $h$ & $S = H= \langle [h] \rangle$ \\  \hline
2     &    &    & $-[1]$ & $\{\pm [1]\}$ \\  \hline
4     & (1, 1) & (99, 31) & $-[746]$ & $\{\pm [1], \pm [746]\}$ \\  \hline
4     & (1, 3) & (99, 6) & $-[2436]$ & $\{\pm [1], \pm [2436]\}$ \\  \hline
6     & (1, 1) & (147, 27) & $-[1712]$ & $\{\pm [1], \pm [1712], \pm [1713]\}$ \\  \hline
6     & (5, 5) & (147, 11) & $-[1543]$ & $\{\pm [1], \pm [1543], \pm [1544]\}$ \\  \hline
12     & (1, 1) & (80, 8) & $-[2286]$ & $\{\pm [1], \pm [746], \pm [1540], \pm [1712], \pm [1713], \pm [2286]\}$ \\  \hline
12     & (1, 5) & (80, 23) & $-[1272]$ & $\{ \pm[1], \pm [526], \pm [746], \pm [1272], \pm [1543], \pm [1544]\}$ \\  \hline
12     & (1, 7) & (80, 29) & $-[2117]$ & $\{ \pm [1], \pm [319], \pm [1712], \pm [1713], \pm [2117], \pm [2436]\}$ \\  \hline
12     & (1, 11) & (80, 14) & $+[3122]$ & $\{ \pm [1], \pm [695], \pm [1543], \pm [1544], \pm [2436], \pm [3122]\}$ \\  \hline
\end{tabular}
  \caption{All rotational first-kind Frobenius circulants with kernel $\Z_{6253}$.}
\label{tab6253}
  \end{center}
\end{table}
}
\end{example}
 
In view of (B) in \S \ref{subsec:motiv}, Theorem \ref{bigcomposite} can be restated as follows. 

\begin{corollary}
\label{core:map}
Let $n = p_1^{e_1}\ldots p_l^{e_l}$ be a positive integer in canonical prime factorization. Let $D = \gcd(p_1 -1, \ldots, p_l-1)$. A first-kind Frobenius circulant $\Ga$ of valency $d$ with kernel $\Z_n$ can be embedded on a closed orientable surface as a balanced regular Cayley map if and only if (i) $n$ is odd and $d$ is an even divisor of $D$, and (ii) $\Ga$ is isomorphic to a circulant constructed in Theorem \ref{bigcomposite}. 
\end{corollary}

\begin{corollary}
\label{core:pp}
Let $p$ be an odd prime, and let $e \ge 1$ and $d \ge 2$ be integers. There exists a rotational first-kind Frobenius circulant on $p^e$ vertices with valency $d$ and kernel $\Z_{p^e}$ if and only if $d$ is an even divisor of $p-1$. Moreover, up to isomorphism such a circulant is unique for every even divisor $d$ of $p-1$,  namely $\Cay(\Z_{p^e}, \langle [\eta^{p^{e-1}(p-1)/d}] \rangle)$, where $\eta$ is a primitive root mod $p^e$. 
\end{corollary}

\begin{proof}
Using the notation in Theorem \ref{bigcomposite}, we have $l = 1$, $D = p-1$, and $b_1 = 1$ in (\ref{CRT3}). The result follows from Theorem \ref{bigcomposite} by noting $\varphi(d)^{l-1} = 1$ for every even divisor $d$ of $p-1$, and by choosing $m_1 = 1$ in (\ref{CRT2}). 
\end{proof}

In the special case when $e=1$, Corollary \ref{core:pp} gives the following: there exists a rotational first-kind Frobenius circulant on $p$ vertices with valency $d$ and kernel $\Z_{p}$ if and only if $d$ is an even divisor of $p-1$; in this case such a circulant is unique and is isomorphic to $\Cay(\Z_{p}, \langle [\eta^{(p-1)/d}] \rangle)$, where $\eta$ is a primitive root mod $p$. It can be verified that this graph is in fact the unique first-kind Frobenius circulant on $p$ vertices with valency $d$ and kernel $\Z_{p}$. The family of such circulants $\Cay(\Z_{p}, \langle [\eta^{(p-1)/d}] \rangle)$, with $d$ running over all even divisors of $p-1$, is precisely the family of all arc-transitive graphs on $p$ vertices as classified in \cite{Chao}. Thus we have: 

\begin{corollary}
\label{core:prime order}
All arc-transitive graphs of prime order are rotational first-kind Frobenius circulants.
\end{corollary}

It follows from \cite[Theorem 2]{TZ-1} and \cite[Theorem 2]{TZ-2}, respectively, that all first-kind Frobenius circulants of valency 4 or 6 are rotational. It would be interesting to explore if this extends to some other valencies.

\subsection{HARTS, or hexagonal meshes}
\label{subsec:harts}

We now explain that HARTS (Hexagonal Architecture for Real-Time Systems), also known as hexagonal mesh, C-wrapped hexagonal mesh or hexagonal interconnection network, belongs to the family \cite{TZ-2} of first-kind Frobenius circulants with kernel $\Z_n$ and valency $6$, which in turn is a subfamily of the family of rotational circulants as classified in Theorem \ref{bigcomposite}. HARTS was proposed \cite{CSK} as a distributed real-time computing system, and its properties were studied in \cite{CSK, DRS, ABF}.

By \cite[Theorem 2]{TZ-2}, a first-kind Frobenius circulant with kernel $\Z_n$ and valency $6$ exists if and only if every prime factor of $n$ is congruent to $1$ mod $6$, and moreover such circulants are precisely $\Cay(\Z_n, \la [a]_n \ra)$ with $a$ running over all 
solutions to the congruence equation $x^2 - x + 1 \equiv 0 \pmod {n}$. In particular, if $n_k = 3k^2 + 3k + 1$ where $k \ge 2$ is an integer, then $3k+2$ is a solution to $x^2 - x + 1 \equiv 0 \pmod {n_k}$. Hence 
$$
TL_{n_k} := \Cay(\Z_{n_k}, \la [3k+2]_{n_k} \ra) = \Cay(\Z_{n_k}, \{\pm [1]_{n_k}, \pm [3k+1]_{n_k}, \pm [3k+2]_{n_k}\})
$$
is a first-kind Frobenius circulant with valency $6$; see \cite{TZ} and \cite[Example 1]{TZ-2}. In \cite{YFMA}, it was proved that $TL_{n_k}$ has diameter $k$, and among all circulants $\Cay(\Z_n, \{\pm [a]_n, \pm [b]_n, \pm [c]_n\})$ of valency $6$ and diameter $k$ such that $a+b+c \equiv 0 \pmod {n}$, $TL_{n_k}$ has the maximum number of vertices.  

The HARTS $H_k$ of size $k$ has diameter $k-1$ and $n_{k-1} = 3k^2 - 3k + 1$ vertices, and is isomorphic \cite{CSK, ABF} to the circulant $\Cay(\Z_{n_{k-1}}, S)$ with $S = \{\pm [k-1]_{n_{k-1}}, \pm [k]_{n_{k-1}}, \pm [2k-1]_{n_{k-1}}\}$. Clearly, $[3k]_{n_{k-1}} \in \Z^*_{n_{k-1}}$. Thus $\Cay(\Z_{n_{k-1}}, S) \cong \Cay(\Z_{n_{k-1}}, S')$ via the isomorphism $[x]_{n_{k-1}} \mapsto [3k]_{n_{k-1}} [x]_{n_{k-1}}$, where $S' = \{\pm [3k]_{n_{k-1}} [k-1]_{n_{k-1}}, \pm [3k]_{n_{k-1}} [k]_{n_{k-1}}, \pm [3k]_{n_{k-1}} [2k-1]_{n_{k-1}}\} = \{\pm [1]_{n_{k-1}}, \pm [3k-1]_{n_{k-1}}, \pm [3k-2]_{n_{k-1}}\}$. It follows that $H_k$ is isomorphic to $TL_{n_{k-1}}$. 

The same conclusion can also be drawn as follows: $H_k$ is isomorphic \cite{ABF} to the EJ network $EJ_{k + (k-1)\omega}$ where $\omega = (1+\sqrt{3}i)/2$, but any EJ network $EJ_{a+b i}$ with $\gcd(a, b) = 1$ is isomorphic \cite[Theorem 5]{TZ-2} to a first-kind Frobenius circulant of valency 6.

\section{A family of rotational non-Frobenius circulants}
\label{sec:family}

The purpose of this section is to prove the following result. 

\begin{theorem}
\label{thm:q}
Let $q = p^e$ with $p$ an odd prime and $e \ge 3$ an integer. Let
\begin{equation}
\label{eq:pr}
\Ga_{q, r} = \Cay(\Z_q, \la [(p-1)^{p^r}] \ra),
\end{equation}
where $r$ is an integer between $0$ and $e-1$. Then $\Ga_{q,r}$ is a connected rotational circulant of valency $2p^{e-r-1}$ that admits $[(p-1)^{p^r}]$ as a complete rotation. The set of fixed points of $[(p-1)^{p^r}]$ is equal to $F = \{[px] \ne [0]: x \in \Z\}$ and is an independent set of $\Ga_{q, r}$. Moreover, $F$ is a vertex-cut of $\Ga_{q, r}$ if and only if $r \ne 0$.
\end{theorem}

The circulants $\Ga_{q, r}$ above cannot be Frobenius, for otherwise the fixed-point set of $[(p-1)^{p^r}]$ would be empty. 
Note that the fixed-point set $F$ of $\Ga_{q, r}$ does not rely on $r$. Note also that $\Ga_{q, e-1}$ is isomorphic to the cycle of $q$ vertices. 

As mentioned in the introduction, it was proved in \cite[Lemma 17]{BKP} that if a Cayley graph has a complete rotation such that the corresponding fixed-point set is an independent set and not a vertex-cut, then its gossiping time can be easily computed. It is known \cite{BKP} that some popular networks, including hypercubes, star graphs and multi-dimensional tori, have this property. Theorem \ref{thm:q} shows that $\{\Ga_{q, 0}: q = p^e,\;\mbox{$p$ an odd prime, $e \ge 3$}\}$ is a new family of rotational Cayley graphs with this property. It was conjectured in \cite{BKP} that the fixed-point set of any complete rotation of a Cayley graph is not a vertex-cut. Lichiardopol \cite{Lich} disproved this by a counterexample. Theorem \ref{thm:q} generalizes his construction to an infinite family of counterexamples, namely $\Ga_{q, r}$ with $1 \le r \le e-1$. In fact, the counterexample in \cite[Section 4]{Lich} is exactly $\Ga_{3^n, 1}$.

An immediate consequence of Theorem \ref{thm:q} and \cite[Lemma 17]{BKP} is that, for any odd prime $p$ and integer $e \ge 3$, the gossiping time of $\Ga_{p^e, 0}$ is equal to $\lceil (p^e - 1)/2p^{e-1}\rceil$, which is quite small compared with the order $p^e$ of $\Ga_{p^e, 0}$.

\subsection{Fixed points of a complete rotation}
\label{subsec:fp}

Before proving Theorem \ref{thm:q} let us briefly discuss basic properties of the fixed-point set of a complete rotation in a rotational Cayley graph.  

Let $\Ga = \Cay(G, S)$ be rotational with a complete rotation $\omega$. Let $H = \langle \omega \rangle \le \Aut(G)$. Then $H$ is regular on $S$ and permutes the elements of $S$ in a cyclic manner. Define \cite[Section 5]{Z} 
$$
X(\omega) = \{x \in G: H_x = 1\} \cup \{1\},
$$
where $H_x$ is the stabilizer of $x$ in $H$. Denote by $F(\omega)$ the set of fixed points of $\omega$. Denote by $\Ga - F(\omega)$ the graph obtained from $\Ga$ by deleting all vertices of $F(\omega)$. 

The first two parts of the following lemma were observed in \cite[Section 5]{Z} under a general setting. We give their proofs for completeness of the present paper.  

\begin{lemma}
\label{le:fx}
With the notation above, we have:
\begin{itemize}
\item[\rm (a)] $F(\omega) = G \setminus X(\omega)$;
\item[\rm (b)] $X(\omega) \setminus \{1\}$ is the union of all $H$-orbits on $G \setminus \{1\}$ with length $|S|$;
\item[\rm (c)] $F(\omega)$ is the union of all other $H$-orbits on $G \setminus \{1\}$, that is, the union of all $H$-orbits whose lengths properly divide $|S|$;
\item[\rm (d)] $F(\omega)$ is not a vertex-cut of $\Ga$ if and only if for every $H$-orbit $x^H$ contained in $X(\omega)$, there exists a path of $\Ga - F(\omega)$ from $1$ to at least one  vertex of $x^H$.
\end{itemize}
\end{lemma}

\begin{proof}
Denote $d = |S|$.

(a) Let $x \in G$. We have: $x \in F(\omega)$ $\Leftrightarrow$ $x \ne 1$ and there exists $i \in \{1, \ldots, d-1\}$ such that $x^{\omega^{i}} = x$ $\Leftrightarrow$ $x \ne 1$ and there exists $i \in \{1, \ldots, d-1\}$ such that $\omega^{i} \in H_x$ $\Leftrightarrow$ $x \ne 1$ and $H_x \ne 1$ $\Leftrightarrow$ $x \in G \setminus X(\omega)$.

(b) Since $H_{x^{h}} = h^{-1} H_x h$  for any $x \in G \setminus \{1\}$ and $h \in H$, $H_x = 1$ if and only if $H_{x^{h}} = 1$. Thus $X(\omega) \setminus \{1\}$ is invariant under $H$ and hence is the union of some $H$-orbits on $G \setminus \{1\}$. By the orbit-stabilizer lemma, an $H$-orbit on $G \setminus \{1\}$ is contained in this union if and only if it has length $d$. 

(c) It can be verified that $F(\omega)$ is invariant under the action of $H$. (In fact, if $x \in F(\omega)$, then there exists $i \in \{1, \ldots, d-1\}$ such that $x^{\omega^{i}} = x$. Thus, for any integer $j$, we have $(x^{\omega^{j}})^{\omega^{i}} = x^{\omega^{j}}$ and so $x^{\omega^{j}} \in F(\omega)$.) This implies that $F(\omega)$ is the union of some $H$-orbits on $G \setminus \{1\}$. Thus, by (a), $F(\omega)$ is the union of all $H$-orbits on $G \setminus \{1\}$ outside $X(\omega)$. Moreover, by (b) and the orbit-stabilizer lemma, $F(\omega)$ is the union of all $H$-orbits whose lengths properly divide $d$.  

(d) The necessity is obvious. To prove the sufficiency, if a path as stated exists, then there exists a path of $\Ga - F(\omega)$ from $1$ to every vertex of $x^H \subseteq X(\omega)$. This is because $H$ fixes $1$, leaves $F(\omega)$ invariant, is a subgroup of $\Aut(\Ga)$, and is transitive on $x^H$. In other words, there exists a path from $1$ to every vertex of $X(\omega)$. Therefore, $\Ga - F(\omega)$ is connected and so $F(\omega)$ is not a vertex-cut of $\Ga$. 
\end{proof}

\subsection{Proof of Theorem \ref{thm:q}}
\label{subsec:family}

Throughout this subsection we assume that $p, q, e, r$ and $h$ are as in Theorem \ref{thm:q} and residue classes are mod $q$. Define 
\begin{equation}
\label{eq:H}
H_{q,r} = \la [h] \ra,\;\,\mbox{where $h = (p-1)^{p^r}$}.
\end{equation}
Since $\gcd(h,p) = 1$, $H_{q,r}$ is a subgroup of $\Z_q^*$.  

Denote by $\ord_{p}(n)$ the exponent of $p$ in the canonical prime factorisation of $n$. In particular, $\ord_{p}(n) = 0$ if $p$ is not a divisor of $n$. In the special case when $p=3$, equality (\ref{eq:a}) below was observed in \cite[Lemma 4.2]{Lich}. 

\begin{lemma}
\label{le:a}
Let $k$ and $s$ be integers with $k \ge 0$ and $1 \le s \le p-1$. 
\begin{itemize}
\item[\rm (a)] If $s$ is even, then
\begin{equation}
\label{eq:a}
\ord_p((p-1)^{p^k s} - 1) = k+1.
\end{equation}
\item[\rm (b)] If $s$ is odd, then
\begin{equation}
\label{eq:a2}
\ord_p((p-1)^{p^k s} + 1) = k+1.
\end{equation}
\end{itemize}
\end{lemma}

\begin{proof}
Denote $a_{k, s} = (p-1)^{p^k s}$. It is obvious that (\ref{eq:a}) and (\ref{eq:a2}) are true when $k=0$. Assume $k \ge 1$ in the sequel. We have 
\begin{equation}
\label{eq:a1}
a_{1, s} = \sum_{j=0}^{ps} \choose{ps}{j} (-1)^j p^{ps - j}. 
\end{equation}

(a) Suppose that $s$ is even. Then the last term on the right-hand side of (\ref{eq:a1}) is $1$, the second last one is $-p^2 s$, and all other terms are divisible by $p^3$. Since $2 \le s \le p-1$, it follows that $\ord_p(a_{1, s} - 1) = 2$ and so (\ref{eq:a}) holds when $k=1$.  

Assume inductively that $\ord_p(a_{k, s} - 1) = k+1$ for some $k \ge 1$ and every even $s$ with $2 \le s \le p-1$. Since for every $i \ge 0$, $p^k si$ is even (as $s$ is even), we have $a_{k, s}^i = \sum_{j=0}^{p^k si} \choose{p^k si}{j} (-1)^j p^{p^k si - j} \equiv 1 \mod p^2$. Thus $\ord_{p}\left(\sum_{i=0}^{p-1} a_{k, s}^i\right) = 1$. Since $a_{k+1, s} - 1 = (a_{k, s} - 1) \cdot \sum_{i=0}^{p-1} a_{k, s}^i$, it follows that $\ord_p(a_{k+1, s} - 1) = \ord_p(a_{k, s} - 1) + \ord_{p}\left(\sum_{i=0}^{p-1} a_{k, s}^i\right) = k+2$, and the proof is complete by induction.

(b) Suppose that $s$ is odd. As in (a), we see from (\ref{eq:a1}) that $\ord_p(a_{1, s} + 1) = 2$.

Assume inductively that $\ord_p(a_{k, s} + 1) = k+1$ for some $k \ge 1$ and every odd $s$ with $1 \le s \le p-2$. Similar to (a), for every $i \ge 0$, we have $a_{k, s}^i \equiv (-1)^i \mod p^2$ and thus $\ord_{p}\left(\sum_{i=0}^{p-1} (-1)^i a_{k, s}^i\right) = 1$. Since $p$ is odd, $a_{k+1, s} + 1 = a_{k, s}^p + 1 = (a_{k, s} + 1) \cdot \sum_{i=0}^{p-1} (-1)^i a_{k, s}^i$. It follows that $\ord_p(a_{k+1, s} + 1) = \ord_p(a_{k, s} + 1) + \ord_{p}\left(\sum_{i=0}^{p-1} (-1)^i  a_{k, s}^i\right) = k+2$, and the proof is complete.
\end{proof}

\begin{lemma}
\label{le:order}
The following hold:
\begin{itemize}
\item[\rm (a)] the order of $[p-1]$ in the multiplicative group $\Z_{q}^*$ is equal to $2p^{e-1}$;
\item[\rm (b)] $|H_{q,r}| = 2p^{e-r-1}$, and moreover
\bea
- H_{q,r} & = & H_{q,r} \non \\
    & = & \{[h^i]: 0 \le i \le 2p^{e-r-1} - 1\} \label{eq:H2} \\
    & = & \{[p^{r+1} k \pm 1]: 0 \le k \le p^{e-r-1}-1\}. \label{eq:H1}
\eea
\end{itemize}
\end{lemma}

\begin{proof}
(a) Let $d$ denote the order of $[p-1]$ in $\Z_{q}^*$. Then $d$ is a divisor of $|\Z_{q}^*| = (p-1)p^{e-1}$. Since $(p-1)^d \equiv 1 \mod q$ by the definition of $d$, we have $(p-1)^d \equiv 1 \mod p$. On the other hand, $p-1 \equiv -1 \mod p$ implies $(p-1)^d \equiv (-1)^d \mod p$. Thus $d$ must be even and so $d = p^k s$ for some integer $k$ with $1 \le k \le e-1$ and even integer $s$ dividing $p-1$. (We have $k \ge 1$ since $(p-1)^s \not \equiv 1 \mod q$ when $s$ is even.) This together with (\ref{eq:a}) implies $k+1 \ge e$. Hence $k=e-1$ and $d = p^{e-1}s$. On the other hand, by (\ref{eq:a}), we have $(p-1)^{2p^{e-1}} \equiv 1 \mod q$. Therefore, $d=2p^{e-1}$.

(b) Since $[p-1]$ has order $2p^{e-1}$ in $\Z_q^*$, the order of $[h]$ in $\Z_{q}^*$ is $2p^{e-r-1}$. That is, $|H_{q,r}| = 2p^{e-r-1}$. Since $r \ge 0$, by (\ref{eq:a2}) we have $(p-1)^{p^{e+r-1}} + 1 \equiv 0 \mod q$. Thus, for every $i$, 
$$
(p-1)^{p^r i} + (p-1)^{p^r (i+p^{e-1})} = (p-1)^{p^r i} \left((p-1)^{p^{e+r-1}} + 1\right) \equiv 0 \mod q.
$$
In other words, $[h^i] = -[h^{i+p^{e-1}}]$, and therefore $-H_{q,r} = H_{q,r}$.

The equality (\ref{eq:H2}) holds because, by (a), $(p-1)^{p^r (j-i)} \not \equiv 1 \mod q$ for $0 \le i < j \le 2p^{e-r-1} - 1$. By (\ref{eq:a2}), we have $(p-1)^{p^r} \equiv -1 \mod p^{r+1}$ and so $(p-1)^{p^r i} \equiv (-1)^i \mod p^{r+1}$. In other words, each $[h^i] = [p^{r+1} k + (-1)^i]$ for some integer $k$. In $\Z_q$, there are exactly $2p^{e-r-1}$ integers of the form $p^{r+1} k \pm 1$, and for distinct $k$ between $0$ and $p^{e-r-1}-1$, these integers are different in $\Z_q$. Therefore, $H_{q,r} = \{[p^{r+1} k \pm 1]: k \in \Z\} = \{[p^{r+1} k \pm 1]: 0 \le k \le p^{e-r-1}-1\}$.
\end{proof}

\begin{lemma}
\label{le:fix}
The set of fixed points of $[h]$ is given by
\begin{equation}
\label{eq:fix}
F = F(h) = \{[px] \ne [0]: x \in \Z\}.
\end{equation}
\end{lemma}

\begin{proof}
Suppose that $[y] \in F(h)$. Then by (\ref{eq:H2}) there exists an $i$, $1 \le i < 2p^{e-r-1}$, such that $[h^i] [y] = [y]$, that is, $\left((p-1)^{p^r i} - 1\right)y \equiv 0 \mod q$. Since $0 < p^r i < 2p^{e-1}$, by part (a) of Lemma \ref{le:order}, we have $\ord_{p}\left((p-1)^{p^r i} - 1\right) < e$. This together with $\left((p-1)^{p^r i} - 1\right)y \equiv 0 \mod q$ implies that $p$ divides $y$ and therefore $F(h) \subseteq \{[px] \ne [0]: x \in \Z\} = \{[px] \in \Z_q: 1 \le x \le p^{e-1} - 1\}$. 

Conversely, consider $[px]$ with $1 \le x \le p^{e-1} - 1$. By (\ref{eq:a}), $p^{e-1}$ divides $(p-1)^{2p^{e-2}} - 1$. Thus $[h^{2p^{e-r-2}}] [px] = [(p-1)^{2p^{e-2}} px] = [px]$ and so $[px] \in F(h)$. Therefore, $\{[px] \in \Z_q: 1 \le x \le p^{e-1} - 1\} \subseteq F(h)$.
\end{proof}

Equivalently, (\ref{eq:fix}) asserts that $X(h)  \setminus \{[0]\} = \Z_q^*$ (as sets). Hence $|X(h) \setminus \{[0]\}| = p^{e-1}(p-1)$. Since $|H_{q,r}| = 2p^{e-r-1}$, it follows from Lemma \ref{le:fx}(b) that $X(h) \setminus \{[0]\}$ is the union of $p^r (p-1)/2$ $H_{q,r}$-orbits on $\Z_q \setminus \{[0]\}$.  

\bigskip
\noindent \textit{Proof of Theorem \ref{thm:q}.}~
Since $-H_{q,r}=H_{q,r}$ (by Lemma \ref{le:order}(b)) and $[1] \in H_{q,r}$, $\Ga_{q,r}$ is an undirected connected circulant graph. Obviously, $\Ga_{q,r}$ has valency $|H_{q,r}| = 2p^{e-r-1}$ and admits $[h]$ (defined in (\ref{eq:H})) as a complete rotation. We saw already that the set of fixed points $F$ of $h$ is given by (\ref{eq:fix}) as claimed. Using this and the definition of $\Ga_{q,r}$, one can see that $F$ is an independent set of $\Ga_{q,r}$. 

We now prove that, if $r \ge 1$, then $F$ is a vertex-cut of $\Ga_{q,r}$. We achieve this by showing that any path in $\Ga_{q,r}$ from $[0]$ to $[p+1]$ must use at least one vertex of $F$. (Note that both $[0]$ and $[p+1]$ are outside $F$.) Suppose to the contrary that there exists a path 
$$
P:  [0], [h^{i_1}], [h^{i_1}] + [h^{i_2}], \ldots, [h^{i_1}] + \cdots + [h^{i_m}] = [p+1]
$$ 
of $\Ga_{q,r}$ whose vertices are all in $X(h)$. As noted in the proof of Lemma \ref{le:order}, each $[h^{i_t}] = [p^{r+1} k_{i_t} + (-1)^{i_t}]$ for some integer $k_{i_t}$. So for any $l$ with $1 \le l \le m$, the $(l+1)$-th vertex of $P$ is $[p^{r+1} k_{l} + a_l]$, where $k_l = \sum_{t=1}^{l} k_{i_t}$ and $a_l = \sum_{t = 1}^l (-1)^{i_t}$. Since all these vertices are in $X(h)$ and are different from $[0]$ (as $P$ is a path and so has no repeated vertices), by (\ref{eq:fix}) we have $a_l \ne 0$, $1 \le l \le m$. Moreover, we can prove by induction that $-(p-1) \le a_l \le p-1$ for each $l$. In fact, this is obviously true when $l=1$. Assume that it is true for some $l$ with $1 \le l < m$. Note that $a_{l+1} = a_l + 1$ or $a_l - 1$. In the former case, we have $a_{l} < p-1$ for otherwise $a_{l+1} = 0$ and $[p^{r+1} k_{l+1} + a_{l+1}] = [p^{r+1} k_{l+1}] \in F$, a contradiction. Hence $-(p-1) \le a_{l+1} \le p-1$. Similarly, one can see that the same inequalities hold when $a_{l+1} = a_l - 1$. Therefore, we have proved $-(p-1) \le a_l \le p-1$ for $1 \le l \le m$. Since the last vertex of $P$ is $[p+1]$, we have $[p^{r+1} k_m + a_m] = [p+1]$. Thus $a_m - 1$ is a multiple of $p$. This together with $-(p-1) \le a_{m} \le p-1$ implies that $a_m = 1$ or $-(p-1)$. If $a_m = 1$, then $[p^{r+1} k_m] = [p]$ and so $p^{r+1}$ divides $p$, which is a contradiction as $r \ge 1$. If $a_m = -(p-1)$, then $[p^{r+1} k_m - p] = [p]$ and so $p^{r+1}$ divides $2p$, which again is impossible as $r \ge 1$. Therefore, we have proved that every path in $\Ga_{q,r}$ from $[0]$ to $[p+1]$ uses at least one vertex of $F$. Consequently, $F$ is a vertex-cut of $\Ga_{q,r}$ when $r \ge 1$. 

Finally, we prove that $F$ is not a vertex-cut of $\Ga_{q,r}$ when $r=0$. Since $H_{q,0} = \{[p k \pm 1]: 0 \le k \le p^{e-1}-1\}$ by (\ref{eq:H1}), one can verify that for $1 \le i < j \le (p-1)/2$, $H_{q,0}[i] \ne H_{q,0}[j]$. Since these $H_{q,0}$-orbits are clearly disjoint from $F$, they must be contained in $X(h)$. Since $X(h) \setminus \{[0]\}$ consists of precisely $(p-1)/2$ $H_{q,0}$-orbits, we conclude that $X(h) \setminus \{[0]\}$ is equal to the union of $H_{q,0}[i]$, $1 \le i \le (p-1)/2$. For each $i$ between $1$ and $(p-1)/2$, there exists a path of $\Ga_{q,0} - F$ from $[0]$ to at least one vertex of $H_{q,0}[i]$, say, the path $[0], [1], [2], \ldots, [i]$. By Lemma \ref{le:fx}(d) we conclude that $F$ is not a vertex-cut of $\Ga_{q,0}$. 
\qed

\medskip
\noindent {\bf Acknowledgements}~ We thank J.~\v{S}ir\'{a}\v{n} for his help on regular Cayley maps and the anonymous referees for their helpful comments. The second author was supported by a Future Fellowship (FT110100629) of the Australian Research Council. 


\bibliographystyle{plain}
\bibliography{bibliography/chordalbib}

\small
{

}

\end{document}